\documentclass[12pt,british]{article}

\usepackage{babel}
\usepackage{amsmath,amsfonts,amssymb,amsthm}
\usepackage{enumerate,framed}
\usepackage{textcomp,url}
\usepackage{eucal}

\usepackage{ucs} 
\usepackage[utf8x]{inputenc} 
\usepackage[T1]{fontenc} 

\usepackage{epsfig}

\usepackage[noBBpl]{mathpazo}
\usepackage{mathabx}

\usepackage[matrix,arrow]{xy}

\theoremstyle{plain}
\newtheorem{thm}{Theorem}

\theoremstyle{remark}
\newtheorem*{rem}{Remark}

\newcommand{\reqref}[1]{(\protect\ref{eq:#1})}
\newcommand{\resec}[1]{Section~\protect\ref{sec:#1}}

\newcommand{\reth}[1]{Theorem~\protect\ref{th:#1}}
\newcommand{\req}[1]{equation~(\protect\ref{eq:#1})}

\makeatletter

\renewcommand{\p@enumii}{}

\def\@enum@{\list{\csname label\@enumctr\endcsname}%
           {\usecounter{\@enumctr}\def\makelabel##1{
\normalfont\ignorespaces\emph{{##1}~}}
\setlength{\labelsep}{3pt}
\setlength{\parsep}{0pt}
\setlength{\itemsep}{0pt}
\setlength{\leftmargin}{0pt}
\setlength{\labelwidth}{0pt}
\setlength{\listparindent}{\parindent}
\setlength{\itemsep}{0pt}
\setlength{\itemindent}{0pt}
\topsep=3pt plus 1pt minus 1 pt}}

\def\@map#1#2[#3]{\mbox{$#1 \colon\thinspace #2 \to #3$}}
\def\map#1#2{\@ifnextchar [{\@map{#1}{#2}}{\@map{#1}{#2}[#2]}}

\makeatother

\newcommand{\rp}{\ensuremath{\mathbb{R}P^2}}
\newcommand{\dt}{\ensuremath{\mathbb D}^{2}}
\newcommand{\Z}{\ensuremath{\mathbb Z}}
\newcommand{\N}{\ensuremath{\mathbb N}}
\newcommand{\St}[1][2]{\ensuremath{\mathbb S}^{#1}}
\newcommand{\Et}{\ensuremath{\mathbb E}^{\,2}}

\renewcommand{\to}{\ensuremath{\longrightarrow}}
\renewcommand{\ker}[1]{\ensuremath{\operatorname{\text{Ker}}\left({#1}\right)}}
\DeclareRobustCommand*{\up}[1]{\textsuperscript{#1}}
\newcommand{\brak}[1]{\ensuremath{\left\{ #1 \right\}}}

\newcommand{\lhra}{\mathrel{\lhook\joinrel\to}}

\newcommand{\ang}[1]{\ensuremath{\left\langle #1\right\rangle}}

\newcommand{\setangr}[2]{\ensuremath{\ang{#1 \,\left\lvert \, #2 \right.}}}
\newcommand{\setangl}[2]{\ensuremath{\ang{\left. #1 \,\right\rvert \, #2}}}
\newcommand{\setr}[2]{\ensuremath{\brak{#1 \,\left\lvert \, #2 \right.}}}
\newcommand{\setl}[2]{\ensuremath{\brak{\left. #1 \,\right\rvert \, #2}}}

\renewcommand{\epsilon}{\varepsilon}
\renewcommand{\th}{\ensuremath{\up{th}}}
\newcommand{\pnm}[1][n]{\ensuremath{P_{{#1}}(M)}}
\newcommand{\sn}[1][n]{\ensuremath{S_{{#1}}}}

\newcommand{\indalp}[3]{\ensuremath{\alpha_{#1,#2,#3}}}
\newcommand{\indbet}[3]{\ensuremath{\beta_{#1,#2,#3}}}
\newcommand{\indgam}[3]{\ensuremath{\gamma_{#1,#2,#3}}}
\newcommand{\indeta}[3]{\ensuremath{\eta_{#1,#2,#3}}}

\newcommand{\texorpdfstring}[2]{#1}

\begin{document}

\title{Braid groups of non-orientable surfaces and the Fadell-Neuwirth short exact sequence}  

\author{Daciberg~Lima~Gon\c{c}alves\\
Departamento de Matem\'atica - IME-USP,\\
Caixa Postal~\textup{66281}~-~Ag.~Cidade de S\~ao Paulo,\\ CEP:~\textup{05314-970} - S\~ao Paulo - SP - Brazil.\\ 
e-mail: \url{dlgoncal@ime.usp.br}\vspace*{4mm}\\
John~Guaschi\\
Laboratoire de Math\'ematiques Nicolas Oresme UMR CNRS~\textup{6139}\\ Universit\'e de Caen BP~\textup{5186},
14032 Caen Cedex, France.\\
e-mail: \url{guaschi@math.unicaen.fr}}

\date{20\up{th}~April~2009}           

\maketitle

\begin{abstract}
\noindent 
Let $M$ be a compact, connected non-orientable surface without boundary and of genus $g\geqslant 3$. We investigate the pure braid groups $P_n(M)$ of $M$, and in particular the possible splitting of the Fadell-Neuwirth
short exact sequence 
\begin{equation*}
1 \to P_m(M \setminus\brak{x_1,\ldots,x_n})
\lhra P_{n+m}(M) \stackrel{p_{\ast}}{\to} P_n(M) \to 1,
\end{equation*}
where $m,n\geqslant 1$, and $p_{\ast}$ is the homomorphism which
corresponds geometrically to forgetting the last $m$ strings. This
problem is equivalent to that of the existence of a section for the
associated fibration $\map{p}{F_{n+m}(M)}[F_n(M)]$  of configuration spaces, defined by $p((x_{1},\ldots,x_{n},x_{n+1}, \ldots, x_{n+m}))= (x_{1}, \ldots, x_{n})$. We show that $p$ and $p_{\ast}$ admit a section if and only if $n=1$. Together with previous results, this completes the resolution of the splitting problem for surfaces pure braid groups.
\end{abstract}

\section{Introduction}\label{sec:intro}

Braid groups of the plane were defined by Artin in~1925~\cite{A1}, and
further studied in~\cite{A2,A3}. Braid groups of surfaces were studied by
Zariski~\cite{Z}, and were later generalised using the
following definition due to Fox~\cite{FoN}. Let $M$ be a compact,
connected surface, and let $n\in\N$. We denote the set of all ordered
$n$-tuples of distinct points of $M$, known as the \emph{$n\th$
configuration space of $M$}, by: 
\begin{equation*}
F_n(M)=\setr{(p_1,\ldots,p_n)}{\text{$p_i\in M$ and $p_i\neq p_j$ if
$i\neq j$}}.
\end{equation*}
Configuration spaces play an important r\^ole in several branches of
mathematics and have been extensively studied, see~\cite{CG,FH} for
example. 

The symmetric group $\sn$ on $n$ letters acts freely on $F_n(M)$ by
permuting coordinates. The corresponding quotient space will be denoted by
$D_n(M)$. Notice that $F_n(M)$ is a regular covering of $D_n(M)$. The
\emph{$n\th$ pure braid group $P_n(M)$} (respectively the \emph{$n\th$
braid group $B_n(M)$}) is defined to be the fundamental group of
$F_n(M)$ (respectively of $D_n(M)$). If $m\in \N$, then we may define
a homomorphism \( \map{p_{\ast}}{\pnm[n+m]}[\pnm] \) induced by the
projection \( \map{p}{F_{n+m}(M)}[F_n(M)] \) defined by \(
p((x_1,\ldots, x_n,x_{n+1}, \ldots, x_{n+m}))= (x_1,\ldots,x_n) \).
Representing $\pnm[n+m]$ geometrically as a collection of $n+m$
strings, $p_{\ast}$ corresponds to forgetting the last $m$ strings.
\textbf{We adopt the convention, that unless explicitly stated, all
homomorphisms \( \pnm[n+m] \to \pnm \) in the text will be this one}. If $M$ is the $2$-disc (or the plane $\mathbb{R}^2$), $B_{n}(M)$ and $P_{n}(M)$ are respectively the classical Artin braid group $B_{n}$ and pure braid group $P_{n}$~\cite{FvB}.

If $M$ is without boundary, Fadell and Neuwirth study the map $p$, and
show (\cite[Theorem~3]{FaN}) that it is a locally-trivial fibration.
The fibre over a point $(x_1,\ldots,x_n)$ of the base space is
$F_m(M\setminus\brak{x_1,\ldots,x_n})$ which we interpret as a
subspace of the total space via the map $\map
i{F_m(M\setminus\brak{x_1,\ldots,x_n})}[F_n(M)]$ defined by
\begin{equation*}
i\left((y_1,\ldots,y_m)\right)=(x_1,\ldots,x_n,y_1,\ldots,y_m).
\end{equation*}
Applying the associated long exact sequence in homotopy, we obtain the \emph{pure
braid group short exact sequence of Fadell and Neuwirth}:
\begin{equation}\label{eq:split}
1 \to P_m\left(M\setminus\brak{x_1,\ldots,x_n}\right) \stackrel{i_{\ast}}{\to} \pnm[n+m] \stackrel{p_{\ast}}{\to}
\pnm \to 1, \tag{\normalfont\textbf{PBS}}
\end{equation}
where $n\geqslant 3$ if $M$ is the sphere $\St$~\cite{Fa,FvB}, $n\geqslant 2$ if $M$ is
the real projective plane $\rp$~\cite{vB}, and $n\geqslant 1$ otherwise~\cite{FaN},
and where $i_{\ast}$ and $p_{\ast}$ are the homomorphisms induced by the maps
$i$ and $p$ respectively. The short exact sequence~\reqref{split} has been widely studied, and may be employed for example to determine presentations of $P_n(M)$ (see \resec{pres}), its centre, and possible torsion.  It was also used in recent work on the structure of the mapping class groups~\cite{PR} and on Vassiliev invariants for surface
braids~\cite{GMP}. 

In the case of $P_{n}$, and taking $m=1$, $\ker{p_{\ast}}$ is a free group of rank $n$. The 
short exact sequence~\reqref{split} splits for all $n\geqslant 1$, and so $P_{n}$ may be described as a repeated semi-direct product of free groups. This decomposition, known as the `combing' operation, is the principal result of Artin's classical theory of braid groups~\cite{A2}, and yields normal forms and a solution to the word problem in $B_{n}$. More recently, it was used by Falk and Randell to study the lower central series and the residual nilpotence of $P_n$~\cite{FR}, and by Rolfsen and Zhu to prove that $P_n$ is bi-orderable~\cite{RZ}.

The problem of deciding whether such a decomposition exists for
braid groups of surfaces is thus fundamental. This was indeed a recurrent
and central question during the foundation of the theory and its
subsequent development during the 1960's~\cite{Fa,FaN,FvB,vB,Bi1}. If
the fibre of the fibration is an Eilenberg-MacLane space then the
existence of a section for $p_{\ast}$ is equivalent to that of a
cross-section for $p$~\cite{Ba,Wh1} (cf.~\cite{GG3}). But with the
exception of the construction of sections in certain cases (for $\St$~\cite{Fa} and the $2$-torus $\mathbb{T}^2$~\cite{Bi1}), no progress on the
possible splitting of~\reqref{split} was recorded for nearly forty
years. In the case of orientable surfaces without boundary of genus at
least two, the question of the splitting of~\reqref{split} which was
posed explicitly by Birman in 1969~\cite{Bi1}, was finally resolved by
the authors, the answer being positive if and only if
$n=1$~\cite{GG1}. 

As for the non-orientable case, the braid groups of $\rp$ were first studied by Van
Buskirk~\cite{vB}, and more recently by Wang~\cite{Wa1} and the authors~\cite{GG3,GG7,GG8}. For $n=1$, we have $P_1(\rp)=B_1(\rp)\cong \Z_2$. Van Buskirk showed that for all $n\geqslant 2$, neither the fibration
$\map{p}{F_n(\rp)}[F_1(\rp)]$ nor the homomorphism
$\map{p_{\ast}}{P_n(\rp)}[P_1(\rp)]$  admit a cross-section (for $p$,
this is a manifestation of the fixed point property of $\rp$), but
that the fibration $\map{p}{F_3(\rp)}[F_2(\rp)]$ admits a
cross-section, and hence so does the corresponding homomorphism $p_{\ast}$. Using coincidence theory, we showed that for $n=2,3$ and $m\geqslant 4-n$, neither the fibration
nor the short exact sequence~\reqref{split} admit a section~\cite{GG3}. In~\cite{GG7}, we gave a complete answer to the splitting problem for $\rp$: if $m,n\in \N$, the homomorphism $\map{p_{\ast}}{P_{n+m}(\rp)}[P_n(\rp)]$ and
the fibration $\map{p}{F_{n+m}(\rp)}[F_n(\rp)]$ admit a section if and only if
$n=2$ and $m=1$. In other words, Van Buskirk's values ($n=2$ and $m=1$) are the only
ones for which a section exists (both on the geometric and the
algebraic level). 

In this paper, we study the splitting problem for compact, connected non-orientable surfaces without boundary and of genus $g\geqslant 3$ (every non-orientable compact surface $M$ without boundary is homeomorphic to the connected sum of $g$ copies of $\rp$, $g\in \N$ being the \emph{genus} of $M$). In the case  of the Klein bottle $\mathbb{K}^2$ ($g=2$), the existence of a non-vanishing vector field implies that there always exists a section, both geometric and algebraic (cf.~\cite{FaN}). Our main result is:
\begin{thm}\label{th:nosplitg}
Let $M$ be a compact, connected, non-orientable surface without boundary of genus $g\geqslant 3$, and let $m,n\in \N$. Then the homomorphism $\map{p_{\ast}}{P_{n+m}(M)}[P_n(M)]$ and
the fibration $\map{p}{F_{n+m}(M)}[F_n(M)]$ admit a section if and only if
$n=1$.
\end{thm}

Applying \reth{nosplitg} and the results of~\cite{GG1,GG7}, we may solve completely the splitting problem for surface pure braid groups:
\begin{thm}\label{th:complete}
Let $m,n\in \N$ and $r\geqslant 0$. Let $N$ be a compact, connected surface possibly with boundary, let $\brak{x_{1},\ldots x_{r}}$ be a finite subset in the interior of $N$, let $M=N\setminus \brak{x_{1},\ldots x_{r}}$, and let $\map{p_{\ast}}{P_{n+m}(M)}[P_{n}(M)]$ be the standard projection.
\begin{enumerate}[(a)]
\item If $r> 0$ or if $M$ has non-empty boundary then $p_{\ast}$ admits a section for all $m$ and $n$.
\item Suppose that $r=0$ and that $M$ is without boundary. Then $p_{\ast}$ admits a section if and only if one of the following conditions holds:
\begin{enumerate}[(i)]
\item $M$ is $\St$, the $2$-torus $\mathbb{T}^2$ or the Klein bottle $\mathbb{K}^2$ (for all $m$ and $n$).
\item $M=\rp$, $n=2$ and $m=1$.
\item $M\neq \rp,\St,\mathbb{T}^2,\mathbb{K}^2$ and $n=1$.
\end{enumerate}
\end{enumerate}
\end{thm}

The rest of the paper is organised as follows. In \resec{pres}, we determine a presentation of $P_n(M)$ (\reth{basicpres}). In \resec{setup}, we study the consequences of the existence of a section in the case $m=1$ and $n\geqslant 2$, \emph{i.e.} $\map{p_{\ast}}{P_{n+1}(M)}[P_n(M)]$. The general strategy of the proof of \reth{nosplitg} is based on the following remark. Suppose that~\reqref{split} splits. If $H$ is any normal subgroup of $P_{n+1}(M)$ contained in $\ker{p_{\ast}}$, the quotiented short exact sequence \(
1\to \ker{p_{\ast}}/H \lhra P_{n+1}(M)/H \to P_n(M)\to 1 \) must also
split. In order to obtain a contradiction, we seek such a subgroup
$H$ for which this short exact sequence does \emph{not} split. However
the choice of $H$ needed to achieve this may be somewhat delicate: if $H$
is too `small', the structure of the quotient $P_{n+1}(\rp)/H$ remains
complicated; on the other hand, if $H$ is too `large', we lose too much
information and cannot reach a conclusion. 
In \resec{proofthm}, we first show that we may reduce to the case $m=1$, and then go on to prove \reth{nosplitg} using the analysis of \resec{setup}. As we shall see in \resec{proofthm}, it suffices to take $H$ to be Abelianisation of $\ker{p_{\ast}}$, in which case the quotient $\ker{p_{\ast}}/H$ is a free Abelian group. We will then deduce \reth{complete}.

\subsection*{Acknowledgements}

This work took place during the visit of the second author to the
Departmento de Mate\-m\'atica do IME-Universidade de S\~ao Paulo during
the periods~14\up{th}~--~29\up{th}~April~2008, 18\up{th}~July~--~8\up{th} August~2008 and 31\up{st}~October~--~10\up{th}~November~2008, and of the visit of the first author to the Laboratoire de Math\'ematiques Nicolas Oresme, Universit\'e de Caen during the period 21\up{st}~November~--~21\up{st}~December~2008. This work was supported by the international Cooperation USP/Cofecub project n\up{o} 105/06, and by the CNRS/CNPq project n\up{o}~21119.

\section{A presentation of \texorpdfstring{$P_n(M)$}{P(M,n)}}\label{sec:pres}

Let $M=M_{g}$ be a compact, connected, non-orientable surface without boundary of genus $g\geqslant 2$.
If $n\in\N$ and $\dt\subseteq M$ is a topological disc, the
inclusion induces a homomorphism $\map
{\iota}{B_n(\dt)}[B_n(M)]$. If $\beta\in B_n(\dt)$ then we shall
denote its image $\iota(\beta)$ simply by $\beta$. For $1\leqslant i<j\leq
n$, we consider the following elements of $P_n(M)$:
\begin{equation*}\label{eq:genspn} 
B_{i,j}= \sigma_i^{-1}\cdots \sigma_{j-2}^{-1} \sigma_{j-1}^2 \sigma_{j-2} \cdots \sigma_i, 
\end{equation*} 
where $\sigma_1,\ldots, \sigma_{n-1}$ are the standard generators of
$B_n(\dt)$. The geometric braid corresponding to $B_{i,j}$ takes the
$i\up{th}$ string once around the $j\up{th}$ string in the positive
sense, with all other strings remaining vertical. For each $1\leq
k\leqslant n$ and $1\leqslant l\leqslant g$, we define a generator $\rho_{k,l}$ which is represented
geometrically by a loop based at the $k\up{th}$ point and which goes
round the $l\up{th}$ twisted handle. These elements are illustrated in
Figure~\ref{fig:gens} that represents $M$ minus a disc.


\begin{figure}[h]
\centering{\includegraphics[scale=0.6,angle=270]{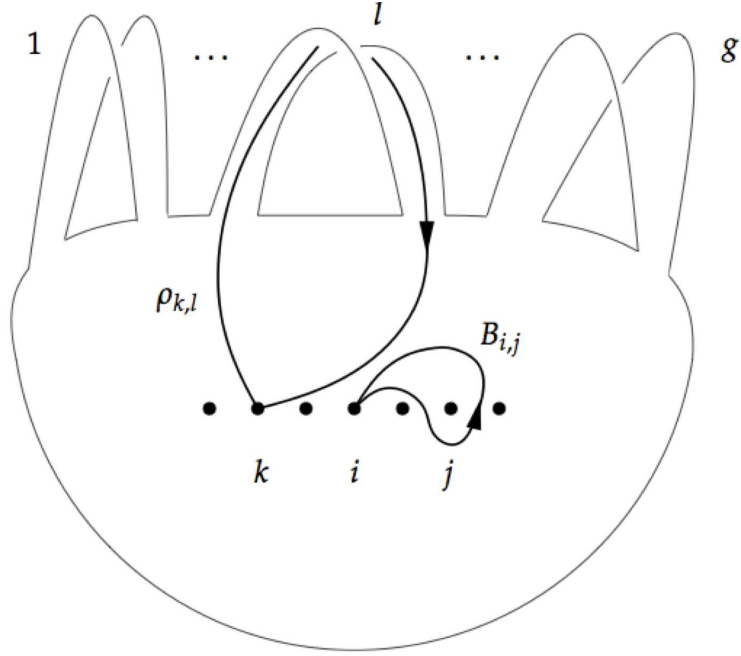}}
\vspace*{-2cm}
\caption{The generators $B_{i,j}$ and $\rho_{k,l}$ of $P_n(M)$, represented geometrically by loops lying in $M$ minus a disc.}
\label{fig:gens}
\end{figure}

A presentation of the braid groups of non-orientable surfaces was originally given by Scott~\cite{S}.  Other presentations were later obtained in~\cite{Be,GM}. In the following theorem, we derive another presentation of $P_n(M)$.

\begin{thm}\label{th:basicpres}
Let $M$ be a compact, connected, non-orientable surface without boundary of genus $g\geqslant 2$, and let $n\in\N$. The following constitutes a presentation of the pure braid group $P_n(M)$:

\begin{enumerate}
\item[\underline{\textbf{generators:}}] $B_{i,j}$, $1\leqslant i<j\leqslant n$, and $\rho_{k,l}$, where $1\leqslant k\leqslant n$ and $1\leqslant l\leqslant g$.
\item[\underline{\textbf{relations:}}]\mbox{}
\begin{enumerate}[(a)]
\item\label{it:rel1} the Artin relations between the $B_{i,j}$ emanating from those of $P_n(\dt)$:
\begin{equation}\label{eq:bijrel}
B_{r,s}B_{i,j}B_{r,s}^{-1}=
\begin{cases}
B_{i,j} & \text{if $i<r<s<j$ or $r<s<i<j$}\\
B_{i,j}^{-1} B_{r,j}^{-1}  B_{i,j} B_{r,j} B_{i,j} & \text{if $r<i=s<j$}\\
B_{s,j}^{-1} B_{i,j} B_{s,j} & \text{if $i=r<s<j$}\\
B_{s,j}^{-1}B_{r,j}^{-1} B_{s,j} B_{r,j} B_{i,j} B_{r,j}^{-1} B_{s,j}^{-1} B_{r,j} B_{s,j}   & \text{if $r<i<s<j$.}
\end{cases}
\end{equation}
\item\label{it:rel2} for all $1\leqslant i<j\leqslant n$ and $1\leqslant k,l \leqslant g$,
\begin{equation}\label{eq:relsrhoik}
\rho_{i,k}\rho_{j,l}\rho_{i,k}^{-1}=
\begin{cases}
\rho_{j,l} & \text{if $k<l$}\\
\rho_{j,k}^{-1} B_{i,j}^{-1}  \rho_{j,k}^2 & \text{if $k=l$}\\
\rho_{j,k}^{-1} B_{i,j}^{-1}\rho_{j,k} B_{i,j}^{-1} \rho_{j,l}
B_{i,j} \rho_{j,k}^{-1} B_{i,j} \rho_{j,k} & \text{if $k>l$}\\
\end{cases}
\end{equation}

\item\label{it:rel3} for all $1\leqslant i\leqslant n$, the `surface relations' $\displaystyle \prod_{l=1}^g \rho_{i,l}^2= B_{1,i}\cdots B_{i-1,i} B_{i,i+1} \cdots B_{i,n}$.

\item\label{it:rel4} for all $1\leqslant i<j\leqslant n$, $1\leqslant k\leqslant n$, $k\neq j$, and $1\leqslant l\leqslant g$,
\begin{equation}\label{eq:relspn}
\rho_{k,l} B_{i,j}\rho_{k,l}^{-1}=
\begin{cases}
B_{i,j} & \text{if $k<i$ or $j<k$}\\
\rho_{j,l}^{-1} B_{i,j}^{-1} \rho_{j,l} & \text{if $k=i$}\\
\rho_{j,l}^{-1} B_{k,j}^{-1} \rho_{j,l} B_{k,j}^{-1} B_{i,j} B_{k,j} \rho_{j,l}^{-1} B_{k,j} \rho_{j,l} & \text{if $i<k<j$}.
\end{cases}
\end{equation}
\end{enumerate}
\end{enumerate}
\end{thm}

\begin{proof}
We apply induction and standard results concerning the presentation of an extension (see~\cite[Theorem~1, Chapter~13]{J}). The proof generalises that of~\cite{GG7} for $\rp$, and is similar in spirit to that of~\cite{S}.

First note that the given presentation is correct for $n=1$ ($P_1(M)= \pi_1(M)$ has a presentation $\setangl{\rho_{1,1}, \dots, \rho_{1,g}}{\prod_{l=1}^g \rho_{1,l}^2=1}$). So let $n\geqslant 1$, and suppose that $P_n(M)$ has the given presentation. Taking $m=1$ in \reqref{split}, we have a short exact sequence:
\begin{equation*}
1 \to \pi_1\left(M\setminus\brak{x_1, \ldots, x_n}, x_{n+1}\right) \to P_{n+1}(M) \stackrel{p_{\ast}}{\to} P_n(M) \to 1.
\end{equation*}
In order to retain the symmetry of the presentation, we take the free group $\ker{p_{\ast}}$ to have the following one-relator presentation:
\begin{equation*}
\setangr{\rho_{n+1,1}, \ldots \rho_{n+1,g}, B_{1,n+1}, \ldots, B_{n,n+1}}{\prod_{l=1}^g \rho_{n+1,l}^2= B_{1,n+1} \cdots B_{n,n+1}}.
\end{equation*}
Together with these generators of $\ker{p_{\ast}}$, the elements $B_{i,j}$, $1\leqslant i<j\leqslant n$, and $\rho_{k,l}$,  $1\leqslant k\leqslant n$ and $1 \leqslant l \leqslant g$, of $P_{n+1}(M)$ (which are coset representatives of the generators of $P_n(M)$)  form the given generating set of $P_{n+1}(M)$. 

There are three classes of relations of $P_{n+1}(M)$ which are
obtained as follows. The first consists of the single relation
$\prod_{l=1}^g \rho_{n+1,l}^2= B_{1,n+1} \cdots B_{n,n+1}$ of $\ker{p_{\ast}}$. The
second class is obtained by rewriting the relators of the quotient in
terms of the coset representatives, and expressing the corresponding
element as a word in the generators of $\ker{p_{\ast}}$. In this way,
all of the relations of $P_n(M)$ lift directly to relations of
$P_{n+1}(M)$, with the exception of the surface relations which
become 
\begin{equation*}
\prod_{l=1}^g \rho_{i,l}^2 = B_{1,i}\cdots B_{i-1,i} B_{i,i+1} \cdots B_{i,n} B_{i,n+1} \quad\text{for all $1\leqslant i\leqslant n$.}
\end{equation*}
Along with the relation of
$\ker{p_{\ast}}$, we obtain the complete set of surface relations
(relations~(\ref{it:rel3})) for $P_{n+1}(M)$.

The third class of relations is obtained by rewriting the conjugates
of the generators of $\ker{p_{\ast}}$ by the coset representatives in
terms of the generators of $\ker{p_{\ast}}$:
\begin{enumerate}[(i)]
\item\label{it:relex1} For all $1\leqslant i<j\leqslant n$ and $1\leqslant l\leqslant n$,
\begin{equation*}
B_{i,j} B_{l,n+1} B_{i,j}^{-1}= 
\begin{cases}
B_{l,n+1} & \text{if $l<i$ or $j<l$}\\
B_{l,n+1}^{-1} B_{i,n+1}^{-1}  B_{l,n+1} B_{i,n+1} B_{l,n+1} & \text{if $l=j$}\\
B_{j,n+1}^{-1} B_{l,n+1} B_{j,n+1} & \text{if $l=i$}\\
B_{j,n+1}^{-1}B_{i,n+1}^{-1} B_{j,n+1} B_{i,n+1} B_{l,n+1} B_{i,n+1}^{-1} B_{j,n+1}^{-1} B_{i,n+1} B_{j,n+1}   & \text{if $i<l<j$.}
\end{cases}
\end{equation*}
\item\label{it:relex2} $B_{i,j} \rho_{n+1,l} B_{i,j}^{-1}= \rho_{n+1,l}$ for all $1\leqslant i<j\leqslant n$ and $1 \leqslant l \leqslant g$.

\item\label{it:relex3} for all $1\leqslant i\leqslant n$ and $1 \leqslant k,l \leqslant g$,
\begin{equation*}
\rho_{i,k} \rho_{n+1,l} \rho_{i,k}^{-1}= 
\begin{cases}
\rho_{n+1,l} & \text{if $k<l$}\\
\rho_{n+1,k}^{-1} B_{i,n+1}^{-1}  \rho_{n+1,k}^2 & \text{if $k=l$}\\
\rho_{n+1,k}^{-1} B_{i,n+1}^{-1}\rho_{n+1,k} B_{i,n+1}^{-1} \rho_{n+1,l}
B_{i,n+1} \rho_{n+1,k}^{-1} B_{i,n+1} \rho_{n+1,k} & \text{if $k>l$.}
\end{cases}
\end{equation*}

\item\label{it:relex4} For all $1\leqslant i,k\leqslant n$ and $1 \leqslant l \leqslant g$,
\begin{equation*}
\rho_{k,l} B_{i,n+1}\rho_{k,l}^{-1}=
\begin{cases}
B_{i,n+1} & \text{if $k<i$}\\
\rho_{n+1,l}^{-1} B_{i,n+1}^{-1} \rho_{n+1,l} & \text{if $k=i$}\\
\rho_{n+1,l}^{-1} B_{k,n+1}^{-1} \rho_{n+1,l} B_{k,n+1}^{-1} B_{i,n+1} B_{k,n+1} \rho_{n+1,l}^{-1} B_{k,n+1} \rho_{n+1,l} & \text{if $i<k$.}
\end{cases}
\end{equation*}
\end{enumerate}
Then relations~(\ref{it:rel1}) for $P_{n+1}(M)$ are obtained from relations~(\ref{it:rel1}) for $P_{n}(M)$ and relations~(\ref{it:relex1}), relations~(\ref{it:rel2}) for $P_{n+1}(M)$ are obtained from relations~(\ref{it:rel2}) for $P_{n}(M)$ and relations~(\ref{it:relex3}), and relations~(\ref{it:rel4}) for $P_{n+1}(M)$ are obtained from relations~(\ref{it:rel4}) for $P_{n}(M)$, and relations~(\ref{it:relex2}) and~(\ref{it:relex4}).
\end{proof}

\section{Analysis of the case \texorpdfstring{$P_{n+1}(M_g)\to P_n(M_g)$}{P(M,n+1)-->P(M,n)}, \texorpdfstring{$n\geqslant 2$}{n>=2}}\label{sec:setup}


For the whole of this section, we suppose that $g\geqslant 3$ and $n\geqslant 2$. By \reth{basicpres}, $P_n(M_g)$ is generated by the union of the $B_{i,j}$, $1\leqslant i<j\leqslant n$, and of the $\rho_{k,l}$, $1\leqslant k\leqslant n$, $1\leqslant l\leqslant g$.
Let us consider the homomorphism $\map{p_{\ast}}{P_{n+1}(M_g)}[P_n(M_g)]$. In this section, we suppose that $p_{\ast}$ admits a section, denoted by $s_{\ast}$. Applying~\reqref{split}, we thus have a split short exact sequence 
\begin{equation}\label{eq:fnmgnp2}
\xymatrix{%
1 \ar[r] & K \ar[r] & P_{n+1}(M_g) \ar@<4pt>[r]^-{p_{\ast}}  & P_n(M_g) \ar@<4pt>@{-->}[l]^-{s_{\ast}} \ar[r] & 1,}
\end{equation}
where $K=\ker{p_{\ast}}=\pi_1(M_g\setminus \brak{x_1,\ldots,x_n},x_{n+1})$ is a free group of rank $n+g-1$, generated by $\brak{B_{1,n+1}, \ldots, B_{n,n+1}, \rho_{n+1,1},\ldots,\rho_{n+1,g}}$, and subject to the relation 
\begin{equation*}
B_{1,n+1} \cdots B_{n,n+1} =\rho_{n+1,1}^2\cdots\rho_{n+1,g}^2.
\end{equation*}
Let $H=[K,K]$ be the commutator subgroup  of $K$. Then $K/H$ is a free Abelian group of rank $n+g-1$. In what follows, we shall not distinguish notationally between the elements of $K$ and those of $K/H$. The quotient group $K/H$ thus has a basis 
\begin{equation}\label{eq:basisb}
\mathcal{B}=\brak{B_{1,n+1},\ldots, B_{n-1,n+1},\rho_{n+1,1}, \ldots,\rho_{n+1,g}},
\end{equation}
and the relation
\begin{equation}\label{eq:surfacerelgen}
B_{n,n+1}= B_{1,n+1}^{-1}\cdots B_{n-1,n+1}^{-1}\rho_{n+1,1}^2\cdots\rho_{n+1,g}^2
\end{equation}
holds in the Abelian group $K/H$. Since $H$ is normal in $P_{n+1}(M_g)$ and $p_{\ast}$ admits a section,  it follows from \req{fnmgnp2} that we have a split short exact sequence 
\begin{equation*}
\xymatrix{%
1 \ar[r] & K/H \ar[r] & P_{n+1}(M_g)/H \ar@<4pt>[r]^-{\overline{p}}  & P_n(M_g) \ar@<3pt>@{-->}[l]^-{\overline{s}} \ar[r] & 1,}
\end{equation*}
where $\overline{p}$ is the homomorphism induced by $p_{\ast}$, and $\overline{s}$ is the induced section.

Consider the subset
\begin{equation*}
\Gamma=\setl{B_{i,j}, \rho_{k,l}}{1\leqslant i<j\leqslant n,\, 1\leqslant k\leqslant n,\, 1\leqslant l\leqslant g}
\end{equation*}
of $P_{n+1}(M_g)/H$. If $g\in \Gamma$ then $\overline{p}(g)=g\in P_{n}(M_{g})$, and so $g^{-1}\ldotp \overline{s}(\overline{p}(g))\in \ker{\overline{p}}=K/H$. Then the integer coefficients $\indalp ijr,\indbet ijq, \indgam klr, \indeta klq$, where $1\leqslant r\leqslant g$ and $1\leqslant q\leqslant n-1$, are (uniquely) defined by the equations:
\begin{equation}\label{eq:defcoeffs}
\left\{\begin{aligned}
\overline{s}\left(B_{i,j}\right) &= B_{i,j} \rho_{n+1,1}^{\indalp ij1} \cdots \rho_{n+1,g}^{\indalp ijg} B_{1,n+1}^{\indbet ij1} \cdots B_{n-1,n+1}^{\indbet ij{n-1}}\\
\overline{s}\left(\rho_{k,l}\right) &= \rho_{k,l} \rho_{n+1,1}^{\indgam kl1} \cdots \rho_{n+1,g}^{\indgam klg} B_{1,n+1}^{\indeta kl1} \cdots B_{n-1,n+1}^{\indeta kl{n-1}}.
\end{aligned}\right.
\end{equation}
There is an equation for each element of $\Gamma$. Most of the elements of $\Gamma$ commute with the elements of the basis $\mathcal{B}$ of $K/H$ given in \req{basisb}. We record the list of conjugates of such elements for later use. In what follows, $1\leqslant i<j \leqslant n$, $1\leqslant k,m\leqslant n$ and $1\leqslant l,r\leqslant g$. In $K/H$, we have
\begin{equation*}
B_{i,j} B_{m,n+1} B_{i,j}^{-1}=B_{m,n+1}
\end{equation*}
(this follows from \req{bijrel} and the fact that the elements $B_{q,n+1}$, $1\leqslant q\leqslant n$, belong to $K/H$ and so commute pairwise), and
\begin{equation*}
B_{i,j} \rho_{n+1,l} B_{i,j}^{-1}=\rho_{n+1,l}
\end{equation*}
by \req{relspn}. Thus $B_{i,j}$ belongs to the centraliser of $K/H$ in $P_{n+1}(M_g)/H$. Also by \req{relspn}, we have
\begin{equation*}
\rho_{k,l} B_{m,n+1} \rho_{k,l}^{-1}=
\begin{cases}
B_{m,n+1} & \!\!\!\!\text{if $k<m$}\\
\rho_{n+1,l}^{-1} B_{m,n+1}^{-1} \rho_{n+1,l}=B_{m,n+1}^{-1} & \!\!\!\!\text{if $k=m$}\\
\rho_{n+1,l}^{-1} B_{k,n+1}^{-1} \rho_{n+1,l} B_{k,n+1}^{-1} B_{m,n+1} B_{k,n+1} \rho_{n+1,l}^{-1} B_{k,n+1} \rho_{n+1,l}\!=\! B_{m,n+1}
& \!\!\!\!\text{if $k>m$,}
\end{cases}
\end{equation*}
so
\begin{equation}
\rho_{k,l} B_{m,n+1} \rho_{k,l}^{-1} = B_{m,n+1}^{1-2\delta_{k,m}}, \label{eq:commbijgen}
\end{equation}
where $\delta_{\cdot,\cdot}$ is the Kronecker delta. By \req{relsrhoik}, we have
\begin{equation*}
\rho_{k,l} \rho_{n+1,r} \rho_{k,l}^{-1}=
\begin{cases}
\rho_{n+1,r} & \text{if $l<r$}\\
\rho_{n+1,l}^{-1} B_{k,n+1}^{-1}  \rho_{n+1,l}^2= \rho_{n+1,l} B_{k,n+1}^{-1} & \text{if $l=r$}\\
\rho_{n+1,l}^{-1} B_{k,n+1}^{-1}\rho_{n+1,l} B_{k,n+1}^{-1} \rho_{n+1,r}
B_{k,n+1} \rho_{n+1,l}^{-1} B_{k,n+1} \rho_{n+1,l} =\rho_{n+1,r} & \text{if $l>r$,}
\end{cases}
\end{equation*}
so
\begin{equation}
\rho_{k,l} \rho_{n+1,r} \rho_{k,l}^{-1}=\rho_{n+1,r} B_{k,n+1}^{-\delta_{l,r}}. \label{eq:commrr}
\end{equation}
Combining equations~\reqref{commbijgen} and~\reqref{commrr}, we obtain
\begin{equation*}
\rho_{k,r}^2 \rho_{n+1,r} \rho_{k,r}^{-2} = \rho_{k,r} \rho_{n+1,r} B_{k,n+1}^{-1} \rho_{k,r}^{-1}= \rho_{n+1,r} B_{k,n+1}^{-1} B_{k,n+1} = \rho_{n+1,r},
\end{equation*}
so
\begin{equation}\label{eq:commrhokr1}
\rho_{k,r} \rho_{n+1,r} \rho_{k,r}^{-1}=\rho_{k,r}^{-1} \rho_{n+1,r} \rho_{k,r}.
\end{equation}
Furthermore, by \req{commbijgen}, $\rho_{k,l}^2$ commutes with $B_{m,n+1}$, and therefore
\begin{equation}\label{eq:commrhokr2}
\rho_{k,l} B_{m,n+1} \rho_{k,l}^{-1}=\rho_{k,l}^{-1} B_{m,n+1} \rho_{k,l}.
\end{equation}
Hence $\rho_{k,l}^2$ also belongs to the centraliser of $K/H$ in $P_{n+1}(M_g)/H$. From equations~\reqref{commbijgen} and~\reqref{commrr}, we obtain the following relations:
\begin{equation}\label{eq:rhojlcomm1}
\rho_{n+1,1}^{\indgam ik1} \cdots \rho_{n+1,g}^{\indgam ikg} \cdot
\rho_{j,l}=\rho_{j,l} \cdot B_{j,n+1}^{-\indgam ikl} \rho_{n+1,1}^{\indgam ik1} \cdots \rho_{n+1,g}^{\indgam ikg}\quad \text{for all $1\leqslant j\leqslant n$},
\end{equation}
and 
\begin{equation*}
B_{1,n+1}^{\indeta ik1} \cdots B_{n-1,n+1}^{\indeta ik{n-1}}
\rho_{j,l}=
\begin{cases}
\rho_{j,l} B_{j,n+1}^{-2\indeta ikj}  B_{1,n+1}^{\indeta ik1} \cdots B_{n-1,n+1}^{\indeta ik{n-1}} & \text{if $1\leqslant j\leqslant n-1$}\\
\rho_{j,l} B_{1,n+1}^{\indeta ik1} \cdots B_{n-1,n+1}^{\indeta ik{n-1}} & \text{if $j=n$.}
\end{cases}
\end{equation*}
Setting $\indeta ikn=0$ for all $1\leq i\leq n$ and $1\leq k\leq g$ yields:
\begin{equation}\label{eq:rhojlcomm2}
B_{1,n+1}^{\indeta ik1} \cdots B_{n-1,n+1}^{\indeta ik{n-1}}\cdot 
\rho_{j,l}= \rho_{j,l} \cdot B_{j,n+1}^{-2\indeta ikj}  B_{1,n+1}^{\indeta ik1} \cdots B_{n-1,n+1}^{\indeta ik{n-1}} \quad \text{for all $1\leqslant j\leqslant n$}.
\end{equation}
Equations~\reqref{rhojlcomm1} and~\reqref{rhojlcomm2} will be employed repeatedly in the ensuing calculations.

We now investigate the images under $\overline{s}$ of some of the relations~(\ref{it:rel2})--(\ref{it:rel4}) of \reth{basicpres} (it turns out that the analysis of the other relations, including~(\ref{it:rel1}), will not be necessary for our purposes).


\begin{enumerate}[(a)]

\item Let $1\leqslant i<j\leqslant n$ and $1\leqslant k,l \leqslant g$. We examine the three possible cases of \req{defcoeffs} (relation~(\ref{it:rel2}) of \reth{basicpres}).
\begin{enumerate}[(i)]
\item $k<l$: then $\rho_{i,k}\rho_{j,l}=\rho_{j,l}\rho_{i,k}$ in $P_n(M_g)$. The respective images under $\overline{s}$ are:
\begin{align*}
\overline{s}\left(\rho_{i,k}\rho_{j,l}\right)=&
\rho_{i,k} \rho_{n+1,1}^{\indgam ik1} \cdots \rho_{n+1,g}^{\indgam ikg} B_{1,n+1}^{\indeta ik1} \cdots B_{n-1,n+1}^{\indeta ik{n-1}}
\rho_{j,l} \rho_{n+1,1}^{\indgam jl1} \cdots \rho_{n+1,g}^{\indgam jlg} B_{1,n+1}^{\indeta jl1} \cdots B_{n-1,n+1}^{\indeta jl{n-1}}\\
=& \rho_{i,k} \rho_{j,l} B_{j,n+1}^{-\indgam ikl-2\indeta ikj}
\rho_{n+1,1}^{\indgam ik1+\indgam jl1} \cdots \rho_{n+1,g}^{\indgam ikg+\indgam jlg}  B_{1,n+1}^{\indeta ik1+\indeta jl1} \cdots B_{n-1,n+1}^{\indeta ik{n-1}+\indeta jl{n-1}}, 
\end{align*}
and 
\begin{align*}
\overline{s}\left(\rho_{j,l}\rho_{i,k}\right)=&
\rho_{j,l} \rho_{n+1,1}^{\indgam jl1} \cdots \rho_{n+1,g}^{\indgam jlg} B_{1,n+1}^{\indeta jl1} \cdots B_{n-1,n+1}^{\indeta jl{n-1}}
\rho_{i,k} \rho_{n+1,1}^{\indgam ik1} \cdots \rho_{n+1,g}^{\indgam ikg} B_{1,n+1}^{\indeta ik1} \cdots B_{n-1,n+1}^{\indeta ik{n-1}}\\
=& \rho_{j,l} \rho_{i,k} B_{i,n+1}^{-\indgam jlk-2\indeta jli}
\rho_{n+1,1}^{\indgam jl1+\indgam ik1} \cdots \rho_{n+1,g}^{\indgam jlg+\indgam ikg} B_{1,n+1}^{\indeta jl1+\indeta ik1} \cdots B_{n-1,n+1}^{\indeta jl{n-1}+\indeta ik{n-1}}.
\end{align*}
The relation $\rho_{i,k}\rho_{j,l}=\rho_{j,l}\rho_{i,k}$ in $P_{n+1}(M_g)$ implies that
$B_{j,n+1}^{-\indgam ikl-2\indeta ikj}=B_{i,n+1}^{-\indgam jlk-2\indeta jli}$.
Comparing coefficients of the elements of $\mathcal{B}$ in $K/H$ (cf.\ \req{basisb}), if $j<n$, we have
\begin{equation}\label{eq:relbcommi}
\left\{\begin{aligned}
&\indgam jlk+2\indeta jli=0 \quad \text{and}\\
&\indgam ikl+2\indeta ikj=0, 
\end{aligned}\right.
\end{equation}
while if $j=n$, applying \req{surfacerelgen} yields
\begin{equation*}
B_{i,n+1}^{\indgam nlk+2\indeta nli}= B_{n,n+1}^{\indgam ikl+2\indeta ikn}= B_{1,n+1}^{-(\indgam ikl+2\indeta ikn)}\cdots B_{n-1,n+1}^{-(\indgam ikl+2\indeta ikn)} \rho_{n+1,1}^{2(\indgam ikl+2\indeta ikn)} \cdots \rho_{n+1,g}^{2(\indgam ikl+2\indeta ikn)},
\end{equation*}
and thus \req{relbcommi} also holds for $j=n$. 
So for all $1\leqslant i<j\leqslant n$ and $1\leqslant k<l \leqslant g$, 
\begin{gather}
\indgam jlk+2\indeta jli=0 \quad \text{and}\label{eq:gamjlk}\\
\indgam ikl+2\indeta ikj=0.\label{eq:etaikj}
\end{gather}

\item $k=l$: then $\rho_{i,k}\rho_{j,k}\rho_{i,k}^{-1}=\rho_{j,k}^{-1} B_{i,j}^{-1}  \rho_{j,k}^2$ in $P_n(M_g)$ for all $1\leqslant i<j\leqslant n$ and $1\leqslant k\leqslant g$. The respective images under $\overline{s}$ are:
\begin{align*}
\overline{s}\left(\rho_{i,k}\rho_{j,k}\rho_{i,k}^{-1}\right)=&
\rho_{i,k} \rho_{n+1,1}^{\indgam ik1} \cdots \rho_{n+1,g}^{\indgam ikg} B_{1,n+1}^{\indeta ik1} \cdots B_{n-1,n+1}^{\indeta ik{n-1}}
\rho_{j,k} \rho_{n+1,1}^{\indgam jk1} \cdots \rho_{n+1,g}^{\indgam jkg} B_{1,n+1}^{\indeta jk1} \cdots B_{n-1,n+1}^{\indeta jk{n-1}}\cdot\\
& B_{n-1,n+1}^{-\indeta ik{n-1}}\cdots B_{1,n+1}^{-\indeta ik1} \rho_{n+1,g}^{-\indgam ikg} \cdots \rho_{n+1,1}^{-\indgam ik1} \rho_{i,k}^{-1}\\
=& \rho_{i,k} \rho_{j,k} B_{j,n+1}^{-\indgam ikk}
\rho_{n+1,1}^{\indgam ik1+\indgam jk1} \cdots \rho_{n+1,g}^{\indgam ikg+\indgam jkg} B_{j,n+1}^{-2\indeta ikj} B_{1,n+1}^{\indeta ik1+\indeta jk1} \cdots B_{n-1,n+1}^{\indeta ik{n-1}+\indeta jk{n-1}}\cdot\\
& \rho_{i,k}^{-1} B_{i,n+1}^{2\indeta iki}
B_{n-1,n+1}^{-\indeta ik{n-1}}\cdots B_{1,n+1}^{-\indeta ik1}  B_{i,n+1}^{\indgam ikk} \rho_{n+1,g}^{-\indgam ikg} \cdots \rho_{n+1,1}^{-\indgam ik1}\\
=& \rho_{i,k} \rho_{j,k} \rho_{i,k}^{-1} B_{j,n+1}^{-\indgam ikk}
B_{i,n+1}^{-(\indgam ikk+\indgam jkk)}
\rho_{n+1,1}^{\indgam ik1+\indgam jk1} \cdots \rho_{n+1,g}^{\indgam ikg+\indgam jkg} 
B_{j,n+1}^{-2\indeta ikj} 
B_{i,n+1}^{-2(\indeta iki+\indeta jki)}\cdot\\
&  B_{1,n+1}^{\indeta ik1+\indeta jk1} \cdots B_{n-1,n+1}^{\indeta ik{n-1}+\indeta jk{n-1}} B_{i,n+1}^{2\indeta iki}
B_{n-1,n+1}^{-\indeta ik{n-1}}\cdots B_{1,n+1}^{-\indeta ik1}  B_{i,n+1}^{\indgam ikk} \rho_{n+1,g}^{-\indgam ikg} \cdots \rho_{n+1,1}^{-\indgam ik1}\\
=& \rho_{i,k} \rho_{j,k} \rho_{i,k}^{-1} \rho_{n+1,1}^{\indgam jk1} \cdots \rho_{n+1,g}^{\indgam jkg} B_{1,n+1}^{\indeta jk1} \cdots B_{n-1,n+1}^{\indeta jk{n-1}}
B_{i,n+1}^{-(2\indeta jki+\indgam jkk)}
B_{j,n+1}^{-(2\indeta ikj+\indgam ikk)} 
\end{align*}
and 
\begin{align*}
\overline{s}\left(\rho_{j,k}^{-1} B_{i,j}^{-1}  \rho_{j,k}^2\right)=& B_{n-1,n+1}^{-\indeta jk{n-1}}\cdots B_{1,n+1}^{-\indeta jk1} \rho_{n+1,g}^{-\indgam jkg} \cdots \rho_{n+1,1}^{-\indgam jk1} \rho_{j,k}^{-1}\cdot\\
& B_{n-1,n+1}^{-\indbet ij{n-1}} \cdots B_{1,n+1}^{-\indbet ij1} \rho_{n+1,g}^{-\indalp ijg} \cdots \rho_{n+1,1}^{-\indalp ij1} B_{i,j}^{-1}\cdot\\
& \rho_{j,k} \rho_{n+1,1}^{\indgam jk1} \cdots \rho_{n+1,g}^{\indgam jkg} B_{1,n+1}^{\indeta jk1} \cdots B_{n-1,n+1}^{\indeta jk{n-1}}
\rho_{j,k} \rho_{n+1,1}^{\indgam jk1} \cdots \rho_{n+1,g}^{\indgam jkg} B_{1,n+1}^{\indeta jk1} \cdots B_{n-1,n+1}^{\indeta jk{n-1}}\\
= & \rho_{j,k}^{-1}B_{i,j}^{-1}
B_{j,n+1}^{2\indeta jkj} B_{n-1,n+1}^{-\indeta jk{n-1}}\cdots B_{1,n+1}^{-\indeta jk1} B_{j,n+1}^{\indgam jkk} \rho_{n+1,g}^{-\indgam jkg} \cdots \rho_{n+1,1}^{-\indgam jk1} 
B_{n-1,n+1}^{-\indbet ij{n-1}}\cdots B_{1,n+1}^{-\indbet ij1}\cdot\\
& \rho_{n+1,g}^{-\indalp ijg} \cdots \rho_{n+1,1}^{-\indalp ij1} \rho_{j,k}^2 
B_{j,n+1}^{-\indgam jkk}
\rho_{n+1,1}^{\indgam jk1} \cdots \rho_{n+1,g}^{\indgam jkg} B_{j,n+1}^{-2\indeta jkj} B_{1,n+1}^{\indeta jk1} \cdots B_{n-1,n+1}^{\indeta jk{n-1}}\cdot\\
& \rho_{n+1,1}^{\indgam jk1} \cdots \rho_{n+1,g}^{\indgam jkg} B_{1,n+1}^{\indeta jk1} \cdots B_{n-1,n+1}^{\indeta jk{n-1}}\\
= & \rho_{j,k}^{-1}B_{i,j}^{-1} \rho_{j,k}^2 
\rho_{n+1,1}^{\indgam jk1-\indalp ij1} \cdots \rho_{n+1,g}^{\indgam jkg-\indalp ijg}  B_{1,n+1}^{\indeta jk1-\indbet ij1} \cdots B_{n-1,n+1}^{\indeta jk{n-1}-\indbet ij{n-1}}.
\end{align*}
Since $\rho_{i,k}\rho_{j,k}\rho_{i,k}^{-1}=\rho_{j,k}^{-1} B_{i,j}^{-1}  \rho_{j,k}^2$ in $P_{n+1}(M_g)$, we obtain
\begin{equation}\label{eq:kequall}
B_{i,n+1}^{-(2\indeta jki+\indgam jkk)}
B_{j,n+1}^{-(2\indeta ikj+\indgam ikk)}=
\rho_{n+1,1}^{-\indalp ij1} \cdots \rho_{n+1,g}^{-\indalp ijg}  B_{1,n+1}^{-\indbet ij1} \cdots B_{n-1,n+1}^{-\indbet ij{n-1}}.
\end{equation}
If $j<n$ then all of the terms in \req{kequall} are expressed in terms of the basis $\mathcal{B}$ of $K/H$ of \req{basisb}, and so for all $1\leqslant i<j\leqslant n-1$,
\begin{gather}
\indalp ijr=0 \quad\text{for all $1\leqslant r\leqslant g$} \label{eq:alij1}\\
\indbet ijs=0 \quad\text{for all $1\leqslant s\leqslant n-1$, $s\notin\brak{i,j}$} \label{eq:betijs}\\
\indbet iji=\indgam jkk+2\indeta jki \label{eq:betiji}\\
\indbet ijj=\indgam ikk+2\indeta ikj.\label{eq:betijj}
\end{gather}
If $j=n$ then substituting for $B_{n,n+1}$ in \req{kequall} using \req{surfacerelgen} and comparing coefficients in $K/H$ of the elements of $\mathcal{B}$
yields
\begin{gather*}
2(2\indeta ikn+\indgam ikk)=\indalp inr \quad\text{for all $1\leqslant r\leqslant g$}\\
(2\indeta ikn+\indgam ikk)=-\indbet ins \quad\text{for all $1\leqslant s\leqslant n-1$, $s\neq i$}\\
2(\indeta ikn-\indeta nki) +(\indgam ikk -\indgam nkk)=-\indbet ini.
\end{gather*}
But $\indeta ikn=0$, so for all $1\leqslant i\leqslant n-1$ and $1\leqslant k\leqslant g$,
\begin{gather}
\indalp inr=2\indgam ikk \quad\text{for all $1\leqslant r\leqslant g$}\label{eq:alinr}\\
\indbet ins=-\indgam ikk \quad\text{for all $1\leqslant s\leqslant n-1$, $s\neq i$}\label{eq:betins}\\
\indbet ini=2\indeta nki +(\indgam nkk-\indgam ikk ).\label{eq:betini}
\end{gather}

\item $k>l$: then $\rho_{i,k}\rho_{j,l}\rho_{i,k}^{-1}=
\rho_{j,k}^{-1} B_{i,j}^{-1}\rho_{j,k} B_{i,j}^{-1} \rho_{j,l} B_{i,j} \rho_{j,k}^{-1} B_{i,j} \rho_{j,k}$ in $P_n(M_g)$. The respective images under $\overline{s}$ are:
\begin{align*}
\overline{s}\left(\rho_{i,k}\rho_{j,l}\rho_{i,k}^{-1}\right)=&\rho_{i,k} \rho_{n+1,1}^{\indgam ik1} \cdots \rho_{n+1,g}^{\indgam ikg} B_{1,n+1}^{\indeta ik1} \cdots B_{n-1,n+1}^{\indeta ik{n-1}}
\rho_{j,l} \rho_{n+1,1}^{\indgam jl1} \cdots \rho_{n+1,g}^{\indgam jlg} B_{1,n+1}^{\indeta jl1} \cdots B_{n-1,n+1}^{\indeta jl{n-1}}\cdot\\
& B_{n-1,n+1}^{-\indeta ik{n-1}} \cdots B_{1,n+1}^{-\indeta ik1} \rho_{n+1,g}^{-\indgam ikg}  \cdots \rho_{n+1,1}^{-\indgam ik1} \rho_{i,k}^{-1}\\
=&  \rho_{i,k} \rho_{j,l}  \rho_{i,k}^{-1} B_{j,n+1}^{-\indgam ikl-2\indeta ikj} B_{i,n+1}^{-\indgam jlk-2\indeta jli}
\rho_{n+1,1}^{\indgam jl1} \cdots \rho_{n+1,g}^{\indgam jlg} B_{1,n+1}^{\indeta jl1} \cdots B_{n-1,n+1}^{\indeta jl{n-1}} 
\end{align*}
and
\begin{align*}
\overline{s}&\left(\rho_{j,k}^{-1} B_{i,j}^{-1}\rho_{j,k} B_{i,j}^{-1} \rho_{j,l} B_{i,j} \rho_{j,k}^{-1} B_{i,j} \rho_{j,k}\right)=\\ &
B_{n-1,n+1}^{-\indeta jk{n-1}} \cdots B_{1,n+1}^{-\indeta jk1} \rho_{n+1,g}^{-\indgam jkg}  \cdots \rho_{n+1,1}^{-\indgam jk1} \rho_{j,k}^{-1}\cdot  B_{n-1,n+1}^{-\indbet ij{n-1}}\cdots B_{1,n+1}^{-\indbet ij1} \rho_{n+1,g}^{-\indalp ijg} \cdots \rho_{n+1,1}^{-\indalp ij1} B_{i,j}^{-1}\cdot\\
&  
\rho_{j,k} \rho_{n+1,1}^{\indgam jk1} \cdots \rho_{n+1,g}^{\indgam jkg} B_{1,n+1}^{\indeta jk1} \cdots B_{n-1,n+1}^{\indeta jk{n-1}}\cdot B_{n-1,n+1}^{-\indbet ij{n-1}}\cdots B_{1,n+1}^{-\indbet ij1} \rho_{n+1,g}^{-\indalp ijg} \cdots \rho_{n+1,1}^{-\indalp ij1} B_{i,j}^{-1}\cdot\\
& \rho_{j,l} \rho_{n+1,1}^{\indgam jl1} \cdots \rho_{n+1,g}^{\indgam jlg} B_{1,n+1}^{\indeta jl1} \cdots B_{n-1,n+1}^{\indeta jl{n-1}}\cdot B_{i,j} \rho_{n+1,1}^{\indalp ij1} \cdots \rho_{n+1,g}^{\indalp ijg} B_{1,n+1}^{\indbet ij1} \cdots B_{n-1,n+1}^{\indbet ij{n-1}}\cdot\\
& B_{n-1,n+1}^{-\indeta jk{n-1}} \cdots B_{1,n+1}^{-\indeta jk1} \rho_{n+1,g}^{-\indgam jkg}  \cdots \rho_{n+1,1}^{-\indgam jk1} \rho_{j,k}^{-1}\cdot B_{i,j} \rho_{n+1,1}^{\indalp ij1} \cdots \rho_{n+1,g}^{\indalp ijg} B_{1,n+1}^{\indbet ij1} \cdots B_{n-1,n+1}^{\indbet ij{n-1}}\cdot\\
& \rho_{j,k} \rho_{n+1,1}^{\indgam jk1} \cdots \rho_{n+1,g}^{\indgam jkg} B_{1,n+1}^{\indeta jk1} \cdots B_{n-1,n+1}^{\indeta jk{n-1}}\\
=& B_{n-1,n+1}^{-\indeta jk{n-1}} \cdots B_{1,n+1}^{-\indeta jk1} \rho_{n+1,g}^{-\indgam jkg}  \cdots \rho_{n+1,1}^{-\indgam jk1} \rho_{j,k}^{-1}B_{i,j}^{-1}\rho_{j,k} B_{j,n+1}^{2\indbet ijj+\indalp ijk} \cdot\\
& 
\rho_{n+1,1}^{\indgam jk1-2\indalp ij1} \cdots \rho_{n+1,g}^{\indgam jkg-2\indalp ijg} B_{1,n+1}^{\indeta jk1-2\indbet ij1} \cdots B_{n-1,n+1}^{\indeta jk{n-1}-2\indbet ij{n-1}}\cdot\\
& B_{i,j}^{-1} \rho_{j,l} \rho_{n+1,1}^{\indgam jl1} \cdots \rho_{n+1,g}^{\indgam jlg} B_{1,n+1}^{\indeta jl1} \cdots B_{n-1,n+1}^{\indeta jl{n-1}}\cdot B_{i,j} \rho_{n+1,1}^{\indalp ij1} \cdots \rho_{n+1,g}^{\indalp ijg} B_{1,n+1}^{\indbet ij1} \cdots B_{n-1,n+1}^{\indbet ij{n-1}}\cdot\\
& B_{n-1,n+1}^{-\indeta jk{n-1}} \cdots B_{1,n+1}^{-\indeta jk1} \rho_{n+1,g}^{-\indgam jkg}  \cdots \rho_{n+1,1}^{-\indgam jk1} \rho_{j,k}^{-1} B_{i,j} \rho_{j,k}
B_{j,n+1}^{-\indalp ijk-2\indbet ijj}\cdot\\
&
\rho_{n+1,1}^{\indalp ij1} \cdots \rho_{n+1,g}^{\indalp ijg} B_{1,n+1}^{\indbet ij1} \cdots B_{n-1,n+1}^{\indbet ij{n-1}}\cdot \rho_{n+1,1}^{\indgam jk1} \cdots \rho_{n+1,g}^{\indgam jkg} B_{1,n+1}^{\indeta jk1} \cdots B_{n-1,n+1}^{\indeta jk{n-1}}\\
=& \rho_{j,k}^{-1}B_{i,j}^{-1}\rho_{j,k}B_{i,j}^{-1} \rho_{j,l} B_{i,j} \rho_{j,k}^{-1} B_{i,j} \rho_{j,k}
B_{j,n+1}^{2\indalp ijl-2\indalp ijk}\cdot B_{1,n+1}^{\indeta jl1} \cdots B_{n-1,n+1}^{\indeta jl{n-1}}
\rho_{n+1,1}^{\indgam jl1} \cdots \rho_{n+1,g}^{\indgam jlg}.
\end{align*}
Since $\rho_{i,k}\rho_{j,l}\rho_{i,k}^{-1}=
\rho_{j,k}^{-1} B_{i,j}^{-1}\rho_{j,k} B_{i,j}^{-1} \rho_{j,l} B_{i,j} \rho_{j,k}^{-1} B_{i,j} \rho_{j,k}$ in $P_{n+1}(M_g)$, we see that
\begin{equation*}
B_{i,n+1}^{-\indgam jlk-2\indeta jli}=  
B_{j,n+1}^{2\indalp ijl-2\indalp ijk+\indgam ikl+2\indeta ikj}.
\end{equation*}
If $j<n$, it follows by comparing coefficients of the elements of $\mathcal{B}$ in $K/H$ that for all $1\leqslant i<j< n$ and $1\leqslant l<k\leqslant g$,
\begin{equation}\label{eq:rhocomm}
\left\{\begin{gathered}
\indgam jlk+2\indeta jli=0\\
2\indalp ijl-2\indalp ijk+\indgam ikl+2\indeta ikj=0.
\end{gathered}\right. 
\end{equation}
If $j=n$ then
\begin{align*}
B_{i,n+1}^{-\indgam nlk-2\indeta nli}=&
B_{n,n+1}^{2\indalp inl-2\indalp ink+\indgam ikl+2\indeta ikn}\\
=&  B_{1,n+1}^{-(2\indalp inl-2\indalp ink+\indgam ikl+2\indeta ikn)} \cdots B_{n-1,n+1}^{-(2\indalp inl-2\indalp ink+\indgam ikl+2\indeta ikn)}\cdot\\
& \rho_{n+1,1}^{2(2\indalp inl-2\indalp ink+\indgam ikl+2\indeta ikn)} \cdots 
\rho_{n+1,g}^{2(2\indalp inl-2\indalp ink+\indgam ikl+2\indeta ikn)},
\end{align*}
and comparing coefficients of the elements of $\mathcal{B}$ in $K/H$, we observe that equations~\reqref{rhocomm} also hold if $j=n$. So for all $1\leqslant i<j \leqslant n$ and $1\leqslant l<k\leqslant g$,
\begin{gather}
\indgam jlk+2\indeta jli=0\label{eq:etajli}\\
2\indalp ijl-2\indalp ijk+\indgam ikl+2\indeta ikj=0.\label{eq:aage}
\end{gather}
\end{enumerate}

\item Let $1\leqslant i\leqslant n$. Then $\displaystyle \prod_{l=1}^g \rho_{i,l}^2= B_{1,i}\cdots B_{i-1,i} B_{i,i+1} \cdots B_{i,n}$ in $P_n(M_g)$ by relation~(\ref{it:rel3}) of \reth{basicpres}. For $1\leqslant l\leqslant g$, note that
\begin{align*}
\overline{s}\left(\rho_{i,l}^2\right) &=\rho_{i,l} \rho_{n+1,1}^{\indgam il1} \cdots \rho_{n+1,g}^{\indgam ilg} B_{1,n+1}^{\indeta il1} \cdots B_{n-1,n+1}^{\indeta il{n-1}} \rho_{i,l} \rho_{n+1,1}^{\indgam il1} \cdots \rho_{n+1,g}^{\indgam ilg} B_{1,n+1}^{\indeta il1} \cdots B_{n-1,n+1}^{\indeta il{n-1}}\\
&= \rho_{i,l}^2
B_{i,n+1}^{-2\indeta ili-\indgam ill}
\rho_{n+1,1}^{2\indgam il1} \cdots \rho_{n+1,g}^{2\indgam ilg} B_{1,n+1}^{2\indeta il1} \cdots B_{n-1,n+1}^{2\indeta il{n-1}}.
\end{align*}
As we saw in equations~\reqref{commrhokr1} and~\reqref{commrhokr2}, $\rho_{i,l}^2$ belongs to the centraliser of $K/H$ in $P_{n+1}(M_{g})/H$, so
\begin{multline*}
\overline{s}\left(\prod_{l=1}^g \rho_{i,l}^2\right) =\\
\left( \prod_{l=1}^g \rho_{i,l}^2 \right) \!\!
\left( \prod_{l=1}^g B_{i,n+1}^{-2\sum_{l=1}^g \indeta ili-\sum_{l=1}^g \indgam ill}
\rho_{n+1,1}^{2\sum_{l=1}^g \indgam il1} \cdots \rho_{n+1,g}^{2\sum_{l=1}^g \indgam ilg} \cdot
B_{1,n+1}^{2\sum_{l=1}^g \indeta il1} \cdots B_{n-1,n+1}^{2\sum_{l=1}^g \indeta il{n-1}} \right)\!. 
\end{multline*}
Further,
\begin{align*}
\overline{s}&\left(B_{1,i}\cdots B_{i-1,i} B_{i,i+1} \cdots B_{i,n}\right) = B_{1,i}\cdots B_{i-1,i} B_{i,i+1} \cdots B_{i,n}\cdot \\ &
\prod_{l=1}^{i-1} \left(\rho_{n+1,1}^{\indalp li1} \cdots \rho_{n+1,g}^{\indalp lig} B_{1,n+1}^{\indbet li1} \cdots B_{n-1,n+1}^{\indbet li{n-1}}\right)
\prod_{l=i+1}^{n} \left(\rho_{n+1,1}^{\indalp il1} \cdots \rho_{n+1,g}^{\indalp ilg} B_{1,n+1}^{\indbet il1} \cdots B_{n-1,n+1}^{\indbet il{n-1}}\right)\\
=& B_{1,i}\cdots B_{i-1,i} B_{i,i+1} \cdots B_{i,n}\cdot\rho_{n+1,1}^{\sum_{l=1}^{i-1} \indalp li1+ \sum_{l=i+1}^{n} \indalp il1} \cdots 
\rho_{n+1,g}^{\sum_{l=1}^{i-1} \indalp lig+ \sum_{l=i+1}^{n} \indalp ilg}\cdot\\
& B_{1,n+1}^{\sum_{l=1}^{i-1}\indbet li1+ \sum_{l=i+1}^{n}\indbet il1} \cdots 
B_{n-1,n+1}^{\sum_{l=1}^{i-1}\indbet li{n-1}+ \sum_{l=i+1}^{n}\indbet il{n-1}}.
\end{align*}
Now in $P_{n+1}(M_g)/H$, $\displaystyle \prod_{l=1}^g \rho_{i,l}^2= B_{1,i}\cdots B_{i-1,i} B_{i,i+1} \cdots B_{i,n}B_{i,n+1}$, hence
\begin{align*}
B_{i,n+1}^{1-2\sum_{l=1}^g \indeta ili-\sum_{l=1}^g \indgam ill} &
\rho_{n+1,1}^{2\sum_{l=1}^g \indgam il1} \cdots \rho_{n+1,g}^{2\sum_{l=1}^g \indgam ilg} B_{1,n+1}^{2\sum_{l=1}^g \indeta il1} \cdots B_{n-1,n+1}^{2\sum_{l=1}^g \indeta il{n-1}}=\\
& \rho_{n+1,1}^{\sum_{l=1}^{i-1} \indalp li1+ \sum_{l=i+1}^{n} \indalp il1} \cdots 
\rho_{n+1,g}^{\sum_{l=1}^{i-1} \indalp lig+ \sum_{l=i+1}^{n} \indalp ilg}\cdot\\
& B_{1,n+1}^{\sum_{l=1}^{i-1}\indbet li1+ \sum_{l=i+1}^{n}\indbet il1} \cdots 
B_{n-1,n+1}^{\sum_{l=1}^{i-1}\indbet li{n-1}+ \sum_{l=i+1}^{n}\indbet il{n-1}}.
\end{align*}
Thus for all $1\leqslant i\leqslant n$,
\begin{multline}
B_{i,n+1}^{1-2\sum_{l=1}^g \indeta ili-\sum_{l=1}^g \indgam ill} =
\rho_{n+1,1}^{\sum_{l=1}^{i-1} \indalp li1+ \sum_{l=i+1}^{n} \indalp il1-2\sum_{l=1}^g \indgam il1} \cdots 
\rho_{n+1,g}^{\sum_{l=1}^{i-1} \indalp lig+ \sum_{l=i+1}^{n} \indalp ilg -2\sum_{l=1}^g \indgam ilg}\cdot\\
 B_{1,n+1}^{\sum_{l=1}^{i-1}\indbet li1+ \sum_{l=i+1}^{n}\indbet il1-2\sum_{l=1}^g \indeta il1} \cdots 
B_{n-1,n+1}^{\sum_{l=1}^{i-1}\indbet li{n-1}+ \sum_{l=i+1}^{n}\indbet il{n-1}- 2\sum_{l=1}^g \indeta il{n-1}}.\label{eq:surface}
\end{multline}
\end{enumerate}

\section{Proofs of Theorems~\ref{th:nosplitg} and~\ref{th:complete}}\label{sec:proofthm}


In this section, we use the calculations of \resec{setup} to prove \reth{nosplitg}, from which we shall deduce \reth{complete}.

\begin{proof}[Proof of \reth{nosplitg}.]
As we mentioned in the Introduction, the existence of an algebraic section for $p_{\ast}$ is equivalent to that of a cross-section for $p$.

The case $n=1$ was treated in Theorem~1 of~\cite{GG1}, using the fact that if $M=M_{g}$, where $g\geq 3$, then $M$ is homeomorphic to the connected sum of one or two copies of $\rp$ with a compact, orientable surface without boundary of genus at least one. 

Conversely, suppose that there exist $m\in \N$ and $n\geq 2$ for which the homomorphism $\map{p_{\ast}}{P_{n+m}(M)}[P_{n}(M)]$ admits a section. We shall argue for a contradiction. By~\cite[Proposition~3]{GG1}, it suffices to consider the case $m=1$. We first analyse the general structure of the coefficients $\indalp ijr,\indbet ijq, \indgam klr, \indeta klq$ defined by \req{defcoeffs}. 

\bigskip
\begin{enumerate}[(a)]
\item\label{it:itema} Taking $j=n$ in \req{etaikj} implies that $\indgam ikl=0$ for all $1\leqslant i\leqslant n-1$ and $1\leqslant k<l\leqslant g$.
\item\label{it:itemb} By \req{aage}, 
\begin{equation*}
\indgam ikl=-2\indeta ikj -2(\indalp ijl-\indalp ijk)
\end{equation*}
for all $1\leqslant i<j\leqslant n$ and $1\leqslant l<k\leqslant g$. Taking $j=n$, we obtain 
\begin{equation*}
\indgam ikl=-2\indeta ikn -2(\indalp inl-\indalp ink)=0
\end{equation*}
since $\indeta ikn=0$ by definition and $\indalp inr=2\indgam i11$ for all $1\leqslant i\leqslant n-1$ and $1\leqslant r \leqslant g$ by \req{alinr}.

It thus follows from~(\ref{it:itema}) and~(\ref{it:itemb}) that
\begin{equation}\label{eq:gamikl0}
\indgam ikl=0 \quad \text{for all $1\leqslant i\leqslant n-1$ and $1\leqslant k,l\leqslant g$, $k\neq l$.}
\end{equation}

\item By \req{alinr}, $\indgam ikk=\frac{1}{2}\indalp in1$ for all $1\leqslant i\leqslant n-1$ and $1\leqslant k\leqslant g$. So
\begin{equation}\label{eq:gamikk}
\indgam ikk=\indgam i11 \quad \text{for all $1\leqslant i\leqslant n-1$ and $1\leqslant k\leqslant g$.}
\end{equation}

\item\label{it:itemd} By \req{etaikj}, for all $1\leqslant k<l \leqslant g$ and $1\leqslant i<j\leqslant n$, we have
\begin{equation*}
\indeta ikj =-\frac{1}{2}\indgam ikl=0,
\end{equation*}
using \req{gamikl0}. So by taking $l=g$ we obtain
\begin{equation*}
\indeta ikj =0 \quad \text{for all $1\leqslant i<j\leqslant n$ and $1\leqslant k\leqslant g-1$.}
\end{equation*}

\item \label{it:iteme} By \req{aage}
\begin{equation*}
\indeta ikj=\frac{1}{2} \left( 2\left( \indalp ijl-\indalp ijk \right) +\indgam ikl \right)
\end{equation*}
for all $1\leqslant i<j\leqslant n$ and $1\leqslant l<k\leqslant g$. But $\indgam ikl=0$ by \req{gamikl0}, and $\indalp ijl-\indalp ijk=0$ by \req{alij1} if $j\leqslant n-1$ and by \req{alinr} if $j=n$. Setting $l=1$, it follows that 
\begin{equation*}
\indeta ikj=0 \quad \text{for all $1\leqslant i<j\leqslant n$ and $2\leqslant k\leqslant g$.}
\end{equation*}
By~(\ref{it:itemd}) and~(\ref{it:iteme}) we thus have
\begin{equation}\label{eq:etaikj0a}
\indeta ikj=0 \quad \text{for all $1\leqslant i<j\leqslant n$ and $1\leqslant k\leqslant g$.}
\end{equation}

\item\label{it:itemf} Suppose that $1\leqslant j<i\leqslant n-1$. Then 
\begin{align*}
\indeta ikj &=-\frac{1}{2}\indgam ikl \quad \text{for all $1\leqslant k<l \leqslant g$, by \req{etajli}}\\
& =0 \quad \text{by \req{gamikl0}.}
\end{align*}
So taking $l=g$, we have $\indeta ikj=0$ for all $1\leqslant k\leqslant g-1$. Further, for all $1\leqslant l< k\leqslant g$,
\begin{align*}
\indeta ikj &= -\frac{1}{2}\indgam ikl \quad \text{by \req{gamjlk}}\\
&= 0 \quad \text{by \req{gamikl0}.}
\end{align*}

Hence it follows from \req{etaikj0a} and~(\ref{it:itemf}) that
\begin{equation}\label{eq:etaikj0}
\indeta ikj=0 \quad \text{for all $1\leqslant i,j \leqslant n-1$, $i\neq j$, and $1\leqslant k\leqslant g$.}
\end{equation}




\item From 
\req{betins}, we obtain
\begin{equation}\label{eq:betins2}
\indbet ins=-\indgam i11 \quad \text{for all $1\leqslant s\leqslant n-1$, $s\neq i$.}
\end{equation}

\item By equations~\reqref{betijj} and~\reqref{etaikj0}, we see that
\begin{equation}\label{eq:gami11a}
\indgam i11=\indbet i{i+1}{i+1}=\cdots = \indbet i{n-1}{n-1} \quad \text{for all $1\leqslant i\leqslant n-2$.}
\end{equation}

\item\label{it:iteml} By equations~\reqref{betiji} and~\reqref{etaikj0}, we obtain
\begin{equation}\label{eq:gami11b}
\indgam i11=\indbet 1i1=\cdots = \indbet {i-1}i{i-1} \quad \text{for all $2\leqslant i\leqslant n-1$.}
\end{equation}
\end{enumerate}

Analysing \req{surface}, we are now able to complete the proof of \reth{nosplitg} as follows.
Let $i\in \brak{1,\ldots, n-1}$. Then the coefficient of $B_{i,n+1}$ yields:
\begin{equation}\label{eq:bin1}
1-2\sum_{l=1}^g \indeta ili -\sum_{l=1}^g \indgam ill= 
\sum_{l=1}^{i-1} \indbet lii +\sum_{l=i+1}^n \indbet ili- 2\sum_{l=1}^g \indeta ili.
\end{equation}
Now
\begin{equation*}
\sum_{l=1}^{i-1} \indbet lii=\sum_{l=1}^{i-1} \indgam l11 \quad\text{by \req{gami11a},}
\end{equation*}
and 
\begin{equation*}
\sum_{l=i+1}^n \indbet ili=\sum_{l=i+1}^n \indgam l11 \quad\text{by \req{gami11b}.}
\end{equation*}
So using \req{gamikk}, \req{bin1} becomes
\begin{equation*}
1-g \indgam i11 = \indbet ini +\sum_{l=1}^{n-1} \indgam l11-\indgam i11.
\end{equation*}
Summing over all $i=1,\ldots,n-1$, and setting $\Delta=\sum_{l=1}^{n-1} \indgam l11$ and $L= \sum_{i=1}^{n-1} \indbet ini$, we obtain
\begin{equation}\label{eq:ng2G}
(n+g-2)\Delta =(n-1)-L.
\end{equation}

Now let $i=n$, and let $k\in \brak{1,\ldots,n-1}$. Since $\indeta nln=0$, the coefficient of $B_{k,n+1}$ in \req{surface} yields:
\begin{align*}
\sum_{l=1}^g \indgam nll-1 &= \sum_{l=1}^{n-1}\indbet lnk -2\sum_{l=1}^g \indeta nlk= \indbet knk +\sum_{\substack{l=1\\ l\neq k}}^{n-1} \indbet lnk-2\sum_{l=1}^g \indeta nlk\\
&= \indbet knk -\sum_{\substack{l=1\\ l\neq k}}^{n-1} \indgam l11 -2\sum_{l=1}^g \indeta nlk \quad \text{by \req{betins2}}\\
&=\indbet knk -(\Delta- \indgam k11) +\sum_{l=1}^g \left( -\indbet knk +\indgam nll -\indgam kll \right) \quad \text{by \req{betini}}\\
&= (1-g)\indbet knk +\indgam k11-\Delta +\sum_{l=1}^g \indgam nll- \sum_{l=1}^g \indgam k11 \quad \text{by \req{gamikk}}\\
&= (1-g)\indbet knk +(1-g)\indgam k11-\Delta +\sum_{l=1}^g \indgam nll.
\end{align*}
Hence
$-1=(1-g)\indbet knk +(1-g)\indgam k11-\Delta$.
Summing over all $k=1,\ldots,n-1$, we obtain
\begin{equation}\label{eq:ng2Ga}
(n+g-2)\Delta=(1-g)L+(n-1).
\end{equation}
Equating equations~\reqref{ng2G} and~\reqref{ng2Ga}, we see that
$(n-1)-L=(1-g)L+(n-1)$.
Since $g\geqslant 3$, it follows that $L=0$, and therefore 
\begin{equation*}
\Delta=\frac{n-1}{(n-1)+(g-1)}
\end{equation*}
by \req{ng2G}. This yields a contradiction to the fact that $\Delta$ is an integer, and thus completes the proof of \reth{nosplitg}.
\end{proof}

\begin{rem}
Although some of the relations derived in~(\ref{it:itema})--(\ref{it:iteml}) do not exist if $n=2$, one may check that the above analysis from \req{bin1} onwards is also valid in this case (with $\Delta=\indgam 111$ and $L=\indbet 121$).
\end{rem}

\begin{proof}[Proof of \reth{complete}.]\mbox{}
\begin{enumerate}[(a)]
\item If $r>0$ then the result follows applying the methods of the proofs of Proposition~27 and Theorem~6 of~\cite{GG1}. If $r=0$ and $M$ has non-empty boundary, let $C$ be a boundary component of $M$. Then $M'=M\setminus C$ is homeomorphic to a compact surface with a single point deleted (which is the case $r=1$), so \reqref{split} splits for $M'$. The inclusion of $M'$ in $M$ not only induces a homotopy equivalence between $M$ and $M'$, but also a homotopy equivalence between their $n\up{th}$ configuation spaces. Therefore their $n\up{th}$ pure braid groups are isomorphic, and the sequence~\reqref{split} for $M$ splits if and only it splits for $M'$. 

\item Suppose that $r=0$ and that $M$ is without boundary. If $M=\St$, $m=1$ and $n\geq 3$ then the statement follows from~\cite{Fa}. The geometric construction of Fadell may be easily generalised to all $m\in \N$. If $n\in\brak{1,2}$, the result is obvious since $P_{n}(\St)$ is trivial.   If $M=\mathbb{T}^2$ or $\mathbb{K}^2$, the fact that $p_{\ast}$ has a section is a consequence of~\cite{FaN} and the fact that $\mathbb{T}^2$ and $\mathbb{K}^2$ admit a non-vanishing vector field. If $M=\rp$ then $p_{\ast}$ admits a section if and only if $n=2$ and $m=1$ by~\cite{GG7}. Finally, if $M\neq \rp,\St,\mathbb{T}^2,\mathbb{K}^2$ then $p_{\ast}$ admits a section if and only if $n=1$ by \reth{nosplitg} for the non-orientable case, and by~\cite{GG1} for the orientable case.\qedhere
\end{enumerate}
\end{proof}

\end{document}